\documentclass{amsart}
\usepackage{amssymb,amsfonts,amsmath,amsthm,amscd,cite,stmaryrd,dsfont,esint,hyperref,pgfplots,standalone,upgreek,xcolor}
\usepackage[title]{appendix}
\usepackage[mathcal]{euscript}

\usepackage{fullpage}

\theoremstyle{plain}
\newtheorem{thm}{Theorem}
\newtheorem{cor}{Corollary}
\newtheorem{lem}[cor]{Lemma}
\newtheorem{prop}[cor]{Proposition}

\theoremstyle{definition}
\newtheorem{defn}[cor]{Definition}

\let\div\undefined{}
\let\P\undefined{}
\DeclareMathOperator{\div}{div}
\DeclareMathOperator{\Lip}{Lip}

\DeclareMathOperator{\sgn}{sgn}

\DeclareMathOperator{\dist}{dist}
\DeclareMathOperator{\cl}{cl}

\DeclareMathOperator{\disc}{disc}

\newcommand{\R}{\mathbb{R}}
\newcommand{\Z}{\mathbb{Z}}

\newcommand{\E}{\mathbb{E}}
\newcommand{\P}{\mathbb{P}}
\renewcommand{\tilde}{\widetilde}

\numberwithin{equation}{section}

\title{On the random G equation with nonzero divergence}
\author{William Cooperman}
\begin{document}
\begin{abstract}
    We prove a quantitative rate of homogenization for the G equation in a random setting with finite range of dependence and nonzero divergence, with explicit dependence of the constants on the Lipschitz norm of the environment. Inspired by work of Burago\textendash{}Ivanov\textendash{}Novikov, the proof uses explicit bounds on the waiting time for the associated metric problem.
\end{abstract}
\maketitle

\section{Introduction}
We consider the behavior, as $\varepsilon \to 0^+$, of the family ${\{u_\varepsilon\}}_{\varepsilon > 0}$ of solutions to the G equation,
\begin{equation}\label{eq:eps-G}
    \begin{cases}
        D_t u^\varepsilon(t, x) - |D_x u^\varepsilon(t, x)| + V(\varepsilon^{-1}x) \cdot D_x u^\varepsilon(t, x) = 0 &\quad \text{in $\R_{> 0} \times \R^d$}\\
        u^\varepsilon(0, x) = u_0(x) &\quad \text{in $\R^d$},
    \end{cases}
\end{equation}
where $d \geq 2$, $V \colon \R^d \to \R^d$ is a random vector field and the initial data $u_0 \colon \R^d \to \R^d$ is Lipschitz. The level sets of $u^\varepsilon$ model a flame front which expands at unit speed in the normal direction while being advected by $V$, which models the wind velocity. When compared with homogenization of other Hamilton-Jacobi equations, the main difficulty with the G equation is that, since we do not assume that $\|V\|_{L^\infty} < 1$, the equation may not be coercive. On the other hand, if $\E[V] = 0$, then the equation is still ``coercive on average'', so we can hope to recover some large-scale controllability.

When $\div V = 0$, the wind cannot form ``traps'' where the flame can be contained, and so a controllability bound holds~\cite{CardSoug}. The main novelty of this paper is a more quantitative controllability bound, which allows for the possibility that $\div V$ is nonzero but small, and rules out the existence of such traps.

Cardaliaguet\textendash{}Souganidis~\cite{CardSoug} proved, under the assumption that the environment $V \in C^{1,1}(\R^d; \R^d)$ is stationary ergodic and divergence-free, that the equation homogenizes; i.e.\ we have the locally uniform convergence of solutions $u^\varepsilon \to \overline{u}$ as $\varepsilon \to 0$ almost surely, where $\overline{u}$ is the solution to the effective equation
\begin{equation}\label{eq:macro-G}
    \begin{cases}
        D_t \overline{u}(t, x) = \overline{H}(D_x\overline{u}(t, x)) &\quad \text{ in $\R_{>0} \times \R^d$}\\
        \overline{u}(0, x) = u_0(x) &\quad \text{ in $\R^d$},
    \end{cases}
\end{equation}
and $\overline{H} \colon \R^d \to \R^d$, called the effective Hamiltonian, is positively homogeneous of degree one and coercive.

Our main result is the following.
\begin{thm}\label{thm:main-homog}
    Let $V \colon \R^d \to \R^d$ be a random Lipschitz vector field which has unit range of dependence and is $\Z^d$-translation invariant. There is a constant $C = C(d) > 0$ such that, if \[ |\div V| \leq C^{-1}{(\|V\|_{C^{0,1}}+1)}^{-C} \] almost surely, then there is a random variable $T_0$, with \[ \E[\exp(C^{-1}{(\|V\|_{C^{0,1}}+1)}^{-C}\log^{3/2} T_0)] \leq C, \] such that
    \begin{equation}\label{eq:main-homog-bound}
        |u^{\varepsilon}(t, x) - \overline{u}(t, x)| \leq C{(\|V\|_{C^{0,1}}+1)}^{C}{(T\varepsilon)}^{1/2}\log^2(\varepsilon^{-1}T)
    \end{equation}
    for all $T \geq \varepsilon T_0$ and $t, |x| \leq T$.
\end{thm}

\subsection{How quantitative is Theorem~\ref{thm:main-homog}?}
There are two main quantitative features of Theorem~\ref{thm:main-homog}: the bound on $|u^\varepsilon - \overline{u}|$, and the random variable $T_0$, which represents how long we must wait before the bound takes effect. As for the former, the exponent $\frac12$ of $(t\varepsilon)$ matches with the best known bound for convergence of the limiting shape in first-passage percolation~\cite{50years}. Indeed, first-passage percolation is an easier problem, since controllability is free and the Hamiltonian is i.i.d., so we cannot hope for a better bound without improving the result for first-passage percolation as well.

As for the bound on $T_0$, we note that the distribution of $T_0$ has subpolynomial tails and therefore all moments of $T_0$ are finite. However, our only bound on the typical value of $T_0$ is \[ \E[T_0] \leq \exp\left(C{(\|V\|_{C^{0,1}}+1)}^C\right). \] We note that while $\|V\|_{C^{0,1}}$ appears to be a random variable, the finite range of dependence assumption implies that it is constant almost surely. The exponential dependence on $\|V\|_{C^{0,1}}$ is an artifact of the fact that the exponent $\frac12$, discussed above, is the tightest possible with our current argument. Indeed, by the same proof it would follow that, if we replace the exponent $\frac12$ in~\ref{eq:main-homog-bound} with an exponent of $\frac12 - \delta$, the corresponding $T_0$ would instead depend polynomially on $\|V\|_{C^{0,1}}$, with the bound \[ \E[T_0] \leq C{(\|V\|_{C^{0,1}}+1)}^{C/\delta}. \]

\subsection{Prior work}
There is a rich body of literature studying homogenization of the G equation and enhancement of the front speed (see~\cite{CardSoug, Burago2, CardNoleSoug, NoleNovi, Caffarelli} for example), so we limit our focus to work most closely related to the current situation. Inspired by Cardaliaguet\textendash{}Souganidis~\cite{CardSoug}, the author~\cite{me} showed that, if $V \in C^{1,1}(\R^d; \R^d)$ has unit range of dependence and is divergence-free, then there is a constant $C = C(d, \|V\|_{C^{1,1}}) > 0$ and a random variable $T_0$ with subpolynomial tail bound $\E[\exp(C^{-1}\log^{3/2} T_0)] \leq C$, such that
\begin{equation}\label{eq:myprevioushomogenization}
    |u^{\varepsilon}(t, x) - \overline{u}(t, x)| \leq C\|u_0\|_{C^{0,1}}{(t\varepsilon)}^{1/2}\log^3(\varepsilon^{-1}t)
\end{equation}for all $t \geq \varepsilon T_0$ and $|x| \leq t$.

Because the bound~\eqref{eq:myprevioushomogenization} is bootstrapped from local controllability estimates by Cardaliaguet\textendash{}Souganidis~\cite{CardSoug}, the dependence of the constant $C$ on $\|V\|_{C^{1,1}}$ was unspecified. Besides, the work of Cardaliaguet\textendash{}Souganidis~\cite{CardSoug} used the divergence-free condition in a critical way, which was necessary under their weaker assumption of stationary ergodicity.

On the other hand, when the environment is periodic instead of random, Cardaliaguet\textendash{}Nolen\textendash{}Souganidis~\cite{CardNoleSoug} proved quantitative homogenization of the G equation without the divergence-free condition. Indeed, they made only the weaker assumption that $|\div V| \leq \varepsilon$ for some $\varepsilon = \varepsilon(d) > 0$, which is related to the constant in the isoperimetric inequality for periodic sets.

We also note that Feldman~\cite{Feldman} extended work of Burago\textendash{}Ivanov\textendash{}Novikov~\cite{Burago2} to prove quantitative estimates on the waiting time in an environment which satisfies a mixing condition in both space and time variables, under the assumption $\div V = 0$.

In this paper, we extend the author's work~\cite{me} to the case where $\div V$ may be nonzero but small and $V$ is only Lipschitz. Along the way, we quantify the dependence of the constant in~\eqref{eq:myprevioushomogenization} on the Lipschitz norm of $V$. The proof adapts an argument of Burago\textendash{}Ivanov\textendash{}Novikov~\cite{Burago}, as well as a new argument to show that, even in the presence of nonzero divergence, the reachable set continues to grow quickly.

\subsection{Definitions, assumptions, and conventions}
We use $C > 0$ to denote a (large) constant which may vary from line to line, but (unless otherwise specified) depends only on the dimension, $d$. For the sake of brevity, we write $\Lip(V)$ to denote the maximum of $1$ and the smallest Lipschitz constant for $V$.

For convenience, we will assume that $V \in C^{1,1}(\R^d; \R^d)$ qualitatively; since every bound we prove depends only on the Lipschitz norm of $V$, this condition can be dropped by approximating $V$ by its mollification. We also assume that there there is zero average drift, i.e. $\E[V] = 0$.

If $E \subseteq \R^d$, we write $\mathcal{G}(E)$ to denote the $\sigma$-algebra generated by $V$ restricted to $E$. That is, $\mathcal{G}(E)$ is the smallest $\sigma$-algebra such that the random variables $V(x)$ are $\mathcal{G}(E)$-measurable for each $x \in E$. We assume that $V$ has unit range of dependence, which means that if $A, B \subseteq \R^d$ are sets with $\dist(A, B) \geq 1$, then $\mathcal{G}(A)$ and $\mathcal{G}(B)$ are independent.

Given $t > 0$ and a measurable function $\alpha \colon [0, t] \to B_1$, define the \textit{controlled path} $X_x^\alpha \colon [0, t] \to \R^d$ to be the solution to the initial-value problem
\begin{equation}\label{controlled-path-ode}
    \begin{cases} \dot{X}_x^\alpha = \alpha + V(X_x^\alpha)\\ X_x^\alpha(0) = x. \end{cases}
\end{equation}
For each $x \in \R^d$, define the \textit{reachable set at time $t$} by
\begin{equation}\label{reachable-set-defn}
    \mathcal{R}_t(x) := \{y \in \R^d \mid \exists\; \alpha \colon [0, t] \to B_1 \text{ such that } X_x^\alpha(t) = y\}.
\end{equation}
Note that this definition still makes sense for $t < 0$, if we interpret $[0, t]$ as $[t, 0]$. For convenience, we also define the sets \[ \mathcal{R}_t^-(x) := \bigcup_{0 \leq s \leq t} \mathcal{R}_s(x) \] for $t \geq 0$ and \[ \mathcal{R}_t^+(x) := \bigcup_{t \leq s \leq 0} \mathcal{R}_s(x) \] for $t \leq 0$. Define the \textit{first passage time} \[ \theta(x, y) := \inf \{t \mid y \in \mathcal{R}_t(x)\}. \]
Finally, if $E \subseteq \R^d$ is a set, we define \[ \mathcal{R}_t(E) = \bigcup_{e \in E} \mathcal{R}_t(e), \] and we do the same for $\mathcal{R}_t^-$ and $\mathcal{R}_t^+$.

\subsection{Acknowledgment}
I would like to thank my advisor, Charles Smart, for suggesting the problem and for many helpful conversations.

\numberwithin{thm}{section}
\numberwithin{cor}{section}
\section{Local waiting time estimates}
In this section, we adapt the proof of Burago\textendash{}Ivanov\textendash{}Novikov~\cite{Burago} to estimate the waiting time for the metric problem associated to the G equation.
\subsection{The incompressible case}
First, we prove that, with high probability, sufficiently large $(d-1)$-dimensional cubes have very little flux.

Let $E(R_1, R_0, \varepsilon)$ be the event that every axis-aligned $(d-1)$-dimensional cube $B$, of radius between $R_0$ and $R_1$, which intersects $Q_{R_1}$ satisfies
\begin{equation}\label{eq:smallaverageepsilon}
    \left|\int_B V(x) \cdot \nu(x) \; \mathrm{d}x\right| \leq \varepsilon |B|,
\end{equation}
where $\nu \colon B \to \R^d$ denotes a unit normal to $B$.

\begin{lem}\label{lem:fluxhelper}
    The event $E(R_1, R_0, \varepsilon)$ has probability at least
    \[
        \P[E(R_1, R_0, \varepsilon)] \geq 1 - C {\left(\frac{R_1\Lip(V)}{\varepsilon}\right)}^d \left(\frac{R_1(1+\|V\|_{L^\infty})}{\varepsilon}\right)\exp\left(\frac{-\varepsilon^2R_0^{d-1}}{C\|V\|_{L^\infty}}\right).
    \]
\end{lem}
\begin{proof}
    Step 1. Let $B$ be an axis-aligned $(d-1)$-dimensional cube of radius $r \geq R_0$. Partition $B$ into at least $C^{-1}r^{d-1}$ equally sized $(d-1)$-dimensional cubes, called $B_1, \dots, B_n$, of radius between $1$ and $2$. For each $i$, the random variable $\int_{B_i} V(x) \cdot \nu(x) \; \mathrm{d}x$ has expectation zero and absolute value at most $C\|V\|_{L^\infty}$. The random variables for non-neighboring cubes are independent, so we can group the sum \[ \int_{B} V(x) \cdot \nu(x) \; \mathrm{d}x = \sum_{i=1}^n \int_{B_i} V(x) \cdot \nu(x) \; \mathrm{d}x \] into $2^{d-1}$ separate sums, each of which contains mutually independent random variable summands, which correspond to non-neighboring cubes. By Azuma's inequality, we conclude that~\eqref{eq:smallaverageepsilon} holds with probability at least \[ 1 - \exp\left(\frac{-\varepsilon^2R_0^{d-1}}{C\|V\|_{L^\infty}}\right). \]

    Step 2. Use the union bound to apply Step 1 to every cube $B$ which has a vertex in ${\left(\frac{\varepsilon}{C\Lip(V)}\Z\right)}^d \cap Q_{R_1 + 1}$ and radius in $\left(\frac{\varepsilon}{C(1+\|V\|_{L^\infty})}\Z\right) \cap [R_0-1, R_1+1]$, and conclude using the Lipschitz bound on $V$, translating and rescaling any cube $B$ so that it has a vertex in this set.
\end{proof}

Next, we show that a subset of $\partial Q_R$ which has small boundary must also have small flux.

\begin{lem}\label{lem:smallflux}
    In the event $E(R_1, R_0, \varepsilon)$, if $D \subseteq \partial Q_R$ has a $(d-2)$-rectifiable boundary for some $0 \leq R \leq R_1$, then
    \[
        \left|\int_D V(x) \cdot \nu(x) \; \mathrm{d}x\right| \leq C\|V\|_{L^\infty}R_0|\partial D| + \varepsilon|\partial Q_R|,
    \]
    where $\nu \colon \partial Q_R \to \R^d$ denotes the outward unit normal to $\partial Q_R$.
\end{lem}
\begin{proof}
    This is exactly Lemma 3.3 from Burago\textendash{}Ivanov\textendash{}Novikov~\cite{Burago}; for completeness, we include the proof here.

    Partition $\partial Q_R$ into at least $C^{-1}{\left(\frac{R}{R_0}\right)}^{d-1}$ equally sized $(d-1)$-dimensional cubes, called $B_1, \dots, B_n$, of radius between $R_0$ and $2R_0$. For each $i$, define $P_i := |\partial D \cap B_i|$ and $S_i := \min \{ |B_i \cap D|, |B_i \setminus D|\}$. The isoperimetric inequality says that \[ S_i \leq CP_i^{(d-1)/(d-2)} \] and the fact that $S_i \subseteq B_i$ implies that \[ S_i \leq CR_0^{d-1}. \] Interpolating between these bounds, we conclude that \[ S_i \leq CP_i R_0. \]

    On the other hand, in the event $E(R_1, R_0, \varepsilon)$ we have \[ \left|\int_{B_i \cap D} V(x) \cdot \nu(x) \; \mathrm{d}x + \int_{B_i \setminus D} V(x) \cdot \nu(x) \; \mathrm{d}x\right| \leq \varepsilon|B_i|. \] Since one of the sets $B_i \cap D$ or $B_i \setminus D$ has measure $S_i$, we conclude that one of the integrals above has absolute value at most $\|V\|_{L^\infty}S_i$, so they both have absolute value at most $\varepsilon|B_i| + \|V\|_{L^\infty}S_i$. We conclude that \[ \left|\int_{B_i \cap D} V(x) \cdot \nu(x) \; \mathrm{d}x\right| \leq \varepsilon|B_i| + C\|V\|_{L^\infty}P_i R_0. \] We conclude by summing over $i$.
\end{proof}

The next lemma shows a weak form of controllability. It can be found in Cardaliaguet\textendash{}Souganidis~\cite{CardSoug} and we include it here as well for completeness.

\begin{lem}\label{lem:cones}
    The reachable set $\mathcal{R}^-_1(x_0)$ contains the cone \[ \left\{x_0 + tv \mid v \in B_{1/2}(V(x_0)), \, 0 \leq t \leq {\left(2\Lip(V)(1+\|V\|_{L^\infty})\right)}^{-1}\right\}. \] Furthermore, suppose that $x_0 \not\in {\left(\mathcal{R}^-_{T+1}(x_0)\right)}^\mathrm{o}$. Then $\mathcal{R}^-_T(x)$ is disjoint from the cone \[ \left\{x_0 + tv \mid v \in B_{1/2}(V(x_0)), \, -{\left(2\Lip(V)(1+\|V\|_{L^\infty})\right)}^{-1} \leq t < 0 \right\}. \]
\end{lem}
\begin{proof}
    Let $v \in B_{1/2}(V(x_0))$ and $t \in (0, {\left(2\Lip(V)(1+\|V\|_{L^\infty})\right)}^{-1})$. For any $t' \in (0, t)$, we have \[ |V(x_0) - V(x_0 + t'v)| \leq t'|v|\Lip(V) \leq \frac12. \] Therefore, $v \in B_1(V(x_0+t'v))$ for all $t \in (0, t)$, so $x_0 + tv \in \mathcal{R}^-_t(x_0) \subseteq \mathcal{R}^-_1(x_0)$, which was the first claim. The contrapositive of the second claim follows by the same argument.
\end{proof}

Finally, we show that, on most of the boundary of the reachable set, the vector field $V$ points toward the interior of the reachable set.

\begin{lem}\label{lem:smallP}
    Let $R > 0$ and $T_0 \geq 0$. Then for every $\varepsilon > 0$, there is some $T_0 < T \leq T_0 + CR^d/\varepsilon$ such that \[ \left|\left\{x \in (\partial \mathcal{R}^-_T(0)) \cap Q_R \mid V(x) \cdot \nu(x) \geq -\frac12\right\}\right| \leq \varepsilon, \] where $|\cdot|$ above denotes the Hausdorff $(d-1)$-measure.
\end{lem}
\begin{proof}
    We use the fact that $t \mapsto |\mathcal{R}^-_t(0) \cap Q_R|$ is Lipschitz with derivative \[ \partial_t |\mathcal{R}^-_t(0)| = \int_{(\partial \mathcal{R}^-_t(0)) \cap Q_R} {(1 + V(x) \cdot \nu(x))}_+ \; \mathrm{d}x \] almost everywhere, where $\nu$ denotes the outward unit normal to $\mathcal{R}^-_t(0)$. Also, for any $t \geq 0$ we have $|\mathcal{R}^-_t(0) \cap Q_R| \leq |Q_R| \leq {(2R)}^d$. The claim follows with $C = 2^{d+1}$ by the mean value theorem.
\end{proof}

All the ingredients for the proof of our local waiting time estimate are now in place.

\begin{thm}\label{thm:maindivfree}
    Suppose that $\div V = 0$ almost surely. Let $W := \inf \{t > 0 \mid \mathcal{R}^-_t(0) \supseteq B_{1/2} \}$. Then for any $\lambda \geq 1$, \[ \P[W \geq \lambda] \leq C\exp\left(-C^{-1}\lambda^{(d-1)/d}{\Lip(V)}^{3-3d}{(1+\|V\|_{L^\infty})}^{3-5d}\right). \]
\end{thm}
\begin{proof}
    Assume $d \geq 3$ (if $d=2$, just add another dimension in which everything is constant). We follow the proof of Burago\textendash{}Ivanov\textendash{}Novikov~\cite{Burago}, keeping track of an extra error term to get a quantitative estimate.

    Let $T > 0$ and $R > 0$. The boundary $\partial (\mathcal{R}^-_T(0) \cap Q_R)$ has two main parts: we define \[ S_R := (\partial \mathcal{R}^-_T(0)) \cap Q_R \] and \[ D_R := \mathcal{R}^-_T(0) \cap (\partial Q_R). \] Further, we let \[ L_R := (\partial \mathcal{R}^-_T(0)) \cap (\partial Q_R). \] Generically, $S_R$ and $D_R$ are $(d-1)$-dimensional and $L_R$ is $(d-2)$-dimensional. Cannarsa\textendash{}Frankowska~\cite{CannarsaFrankowska} proved that the boundary of the reachable set $\mathcal{R}^-_T(0)$ is $C^{1,1}$ everywhere except at the origin, so, as in Burago\textendash{}Ivanov\textendash{}Novikov~\cite{Burago} (see the remark after Lemma~2.4), it can be equipped with a continuous unit normal and the divergence theorem holds.

    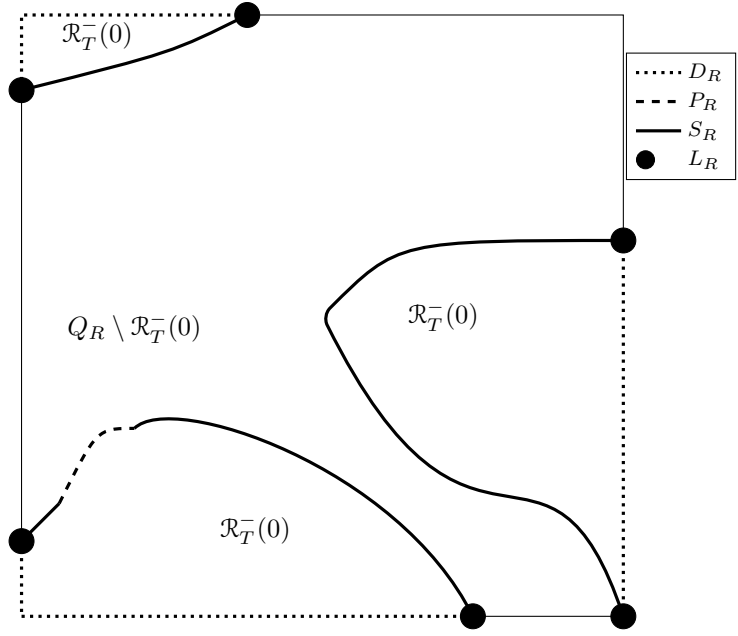
\begin{figure}
        \centering

\newenvironment{customlegend}[1][]{%
    \begingroup
    \csname pgfplots@init@cleared@structures\endcsname
    \pgfplotsset{#1}%
}{%
    \csname pgfplots@createlegend\endcsname
    \endgroup
}%

\def\addlegendimage{\csname pgfplots@addlegendimage\endcsname}

\begin{tikzpicture}
    \draw[very thick, dotted] (0,0) -- (6,0);
    \draw (6,0) -- (8,0);
    \draw[very thick, dotted] (8,0) -- (8,5);
    \draw (8,5) -- (8,8) -- (3, 8);
    \draw[very thick, dotted] (3, 8) -- (0,8) -- (0,7);
    \draw (0,7) -- (0,1);
    \draw[very thick, dotted] (0,1) -- (0,0);
    \draw[very thick, rounded corners] (8,0) ..controls (7, 3) and (6, 0) .. (4,4) node[right=1cm] {$\mathcal{R}^-_T(0)$} ..controls (5,5) .. (8,5);
    \draw[very thick, rounded corners] (1.5,2.5) node[below right=1cm] {$\mathcal{R}^-_T(0)$} ..controls (2,3) and (5,2) .. (6,0);
    \draw[very thick, rounded corners] (0,7) node[above right=0.4cm] {$\mathcal{R}^-_T(0)$} ..controls (2,7.5) .. (3,8);
    \draw[very thick, dashed] (0.5,1.5) ..controls (1,2.5) .. (1.5,2.5) node[above=1cm] {$Q_R \setminus \mathcal{R}^-_T(0)$};
    \draw[very thick] (0,1) -- (0.5,1.5);
    \node[style={draw, circle, fill}] at (6, 0) {};
    \node[style={draw, circle, fill}] at (8, 0) {};
    \node[style={draw, circle, fill}] at (8, 5) {};
    \node[style={draw, circle, fill}] at (0, 1) {};
    \node[style={draw, circle, fill}] at (0, 7) {};
    \node[style={draw, circle, fill}] at (3, 8) {};

    \begin{customlegend}[legend cell align=left,
            legend entries={
                $D_R$,
                $P_R$,
                $S_R$,
                $L_R$
            },
            legend style={at={(9.5,7.5)},font=\footnotesize}]
            \addlegendimage{very thick, dotted}
        \addlegendimage{very thick, dashed}
        \addlegendimage{very thick}
        \addlegendimage{legend image code/.code={\node [draw, circle, fill] {};}}
    \end{customlegend}
\end{tikzpicture}
        \caption{Parts of the reachable set $\mathcal{R}^-_T(0)$, the cube $Q_R$, and boundaries $S_R$, $D_R$, $L_R$, and $P_R$.}
    \end{figure}

    If $x \in S_R$, let $\nu(x)$ denote the outward unit normal to $\mathcal{R}^-_T(0)$ and if $x \in \partial Q_R$, let $\nu(x)$ denote the outward unit normal to $Q_R$. We also define the subset $P_R$ of $S_R$ to be the part of the boundary of the reachable set which is growing at a speed of more than $\frac12$: \[ P_R := \left\{x \in S_R \mid V(x) \cdot \nu(x) \geq -\frac12\right\}. \] Integrate $\div V = 0$ in $\mathcal{R}^-_T(0) \cap Q_R$ to find
    \begin{equation}\label{eq:divergence}
        |S_R \setminus P_R| \leq 2 \int_{P_R \cup D_R} V(x) \cdot \nu(x) \; \mathrm{d}x \leq 2\|V\|_{L^\infty}(|P_R| + |D_R|),
    \end{equation}
    where $|\cdot|$ denotes the Hausdorff measure of appropriate dimension (here it's $d-1$).

    On the other hand, the co-area inequality yields
    \begin{equation}\label{eq:coarea}
        |S_R| \geq \int_0^R |L_r| \; \mathrm{d}r.
    \end{equation}

    We also apply the isoperimetric inequality in $\partial Q_R$:
    \begin{equation}
        \min(|D_R|, |\partial Q_R \setminus D_R|) \leq C|L_R|^{\frac{d-1}{d-2}},
    \end{equation}
    and so the divergence theorem applied to $Q_R$ yields
    \begin{equation}\label{eq:isoperimetric}
        \left|\int_{D_R} V(x) \cdot \nu(x) \; \mathrm{d}x\right| \leq C|L_R|^{\frac{d-1}{d-2}}.
    \end{equation}

    Combining~\eqref{eq:divergence},~\eqref{eq:coarea}, and~\ref{eq:isoperimetric} yields
    \begin{equation}\label{eq:initial-L-bound}
        |L_R| \geq C^{-1}{\left(\int_0^R |L_r| \; \mathrm{d}r \, - \, (1+2\|V\|_{L^\infty})|P_R|\right)}^{\frac{d-2}{d-1}}.
    \end{equation}

    Everything so far has only used incompressibility and boundedness of $V$ in $L^\infty$, and applies for all $T, R > 0$. From now on, we start selecting parameters to show that $B_{1/2} \subseteq \mathcal{R}^-_T(0)$. For the rest of the proof, we assume for contradiction that $B_{\frac12} \not\subseteq \mathcal{R}^-_{T+1}(0)$.

    First, choose $\varepsilon > 0$ such that $\int_0^1 |L_r| \; dr \geq \varepsilon$. By Lemma~\ref{lem:cones} and the isoperimetric inequality, we can choose
    \begin{equation}\label{eq:epsdef}
        \varepsilon := C_0^{-1}{\left(\Lip(V)(1+\|V\|_{L^\infty})\right)}^{1-d},
    \end{equation}
    where $C_0 = C_0(d) > 2^{d-1}$ will be chosen by the end of the proof.

    Next, choose \[ R_1 := {\left(\frac{\lambda\varepsilon}{1+\|V\|_{L^\infty}}\right)}^{1/d} \] and \[ R_0 := \frac{R_1}{C_1{(1+\|V\|_{L^\infty})}^{2}}, \] where $C_1 = C_1(d) > 0$ will also be chosen by the end of the proof. By Lemma~\ref{lem:fluxhelper}, the event $E(R_1, R_0, \varepsilon)$ has probability at least
    \begin{align*}
        \P[E(R_1, R_0, \varepsilon)] &\geq 1-CR_0^{d+1}\|V\|_{C^{0,1}}^C\exp\left(-C^{-1}\lambda^{(d-1)/d}\varepsilon^{2+(d-1)/d}{(1+\|V\|_{L^\infty})}^{-2d}\right)\\
        &\geq 1-C\exp\left(-C^{-1}\lambda^{(d-1)/d}{\Lip(V)}^{3-3d}{(1+\|V\|_{L^\infty})}^{3-5d}\right).
    \end{align*}

    Work in the event $E(R_1, R_0, \varepsilon)$. By Lemma~\ref{lem:smallP}, we can choose $1 \leq T \leq CR_1^d(1+\|V\|_{L^\infty})\varepsilon^{-1} \leq C\lambda$ such that $(1+2\|V\|_{L^\infty})|P_R| \leq \frac{\varepsilon}{2}$. Since we assume that $B_{1/2} \not\in \mathcal{R}^-_{T+1}(0)$, our choice of $T$ does not affect $\varepsilon$. Plug our choice of $\varepsilon$ and $T$ into~\eqref{eq:initial-L-bound} to see that \[ 
    \frac{d}{dR} \int_0^R |L_r| \; \mathrm{d}r \geq C^{-1}{\left(\int_0^R |L_r| \; \mathrm{d}r\right)}^\frac{d-2}{d-1}
    \] for all $1 \leq R \leq R_1$ and \[ \int_0^1 |L_r| \; \mathrm{d}r \geq \frac{\varepsilon}{2}, \] which implies
    \begin{equation}\label{eq:polygrowth}
        \int_0^R |L_r| \; \mathrm{d}r \geq C^{-1}{(R-1)}^{d-1}
    \end{equation}
    for all $1 \leq R \leq R_1$. We make a note that the constant $C$ in~\ref{eq:polygrowth} does not depend on $C_0$ or $C_1$.

    Combining~\eqref{eq:polygrowth} with~\eqref{eq:divergence} and~\eqref{eq:coarea}, we conclude that
    \begin{equation}\label{eq:fluxgrowth}
        C^{-1}{(R-1)}^{d-1} \leq (1+2\|V\|_{L^\infty})|P_R| + \int_{D_R} V(x) \cdot \nu(x) \; \mathrm{d}x.
    \end{equation}

    Apply Lemma~\ref{lem:smallflux} to $D_R$ and combine with~\eqref{eq:fluxgrowth} to obtain
    \begin{equation}
        C^{-1}{(R-1)}^{d-1} \leq (1+2\|V\|_{L^\infty})|P_R| + C(1+\|V\|_{L^\infty})R_0|L_R| + 2d\varepsilon R^{d-1}.
    \end{equation}
    As long as $\varepsilon \leq \frac{1}{2^{d+2}dC}$ and $|P_R| \leq \frac{1}{4C(1+\|V\|_{L^\infty})}$, which we ensure by choosing $C_0 > 0$ sufficiently large in~\eqref{eq:epsdef}, then for $R \geq 2$ we have
    \begin{equation}
        R^{d-1} \leq CR_0(1+\|V\|_{L^\infty})|L_R|.
    \end{equation}

    We integrate and apply~\eqref{eq:coarea} to conclude that
    \begin{equation}
        |S_R| \geq \frac{R^d-1}{CR_0(1+\|V\|_{L^\infty})}
    \end{equation}
    for every $2 \leq R \leq R_1$.

    At $R = R_1$, this yields \[ |S_{R_1}| \geq C_1 C^{-1} (1+\|V\|_{L^\infty}){R_1}^{d-1}. \]
    To conclude, we choose $C_1$ large enough so that $|S_{R_1}| \geq 2(1+\|V\|_{L^\infty})(1 + d2^d R_1^{d-1})$, which contradicts~\eqref{eq:divergence}, as $|P_{R_1}| \leq 1$ and $D_{R_1} \subseteq \partial Q_{R_1}$ and hence $|D_{R_1}| \leq |\partial Q_{R_1}| = d2^d R_1^{d-1}$.
\end{proof}

\subsection{The compressible case}
Next, we adapt the proof in the incompressible case to allow $|\div V|$ to be nonzero but small.

\begin{prop}\label{prop:mainsmalldiv}
    Let $W := \inf \{t > 0 \mid \mathcal{R}^-_t(0) \supseteq B_{1/2} \}$. For each $p > 0$, there is $\varepsilon(\Lip(V), \|V\|_{L^\infty}, p, d) > 0$ such that, if $|\div V| \leq \varepsilon$ almost surely, then \[ \P\left[W \geq C {\Lip(V)}^{3d}{(1+\|V\|_{L^\infty})}^{5d+4} {\left(\log p^{-1}\right)}^{d/(d-1)}\right] \leq p. \] Furthermore, we can choose \[ \varepsilon \geq C^{-1}{\Lip(V)}^{-3}{(1+\|V\|_{L^\infty})}^{-6} {\left(\log p^{-1}\right)}^{1/(d-1)}. \]
\end{prop}
\begin{proof}
    The proof is nearly identical to the proof of Theorem~\ref{thm:maindivfree}. The only difference is the the addition of $\varepsilon C R^d$ error terms in~\eqref{eq:divergence}
    \begin{equation}\label{eq:divergence-new}
        |S_R \setminus P_R| \leq 2 \int_{P_R \cup D_R} V(x) \cdot \nu(x) \; \mathrm{d}x + \varepsilon C R^d \leq 2\|V\|_{L^\infty}(|P_R| + |D_R|) + \varepsilon C R^d,
    \end{equation}
    and~\eqref{eq:isoperimetric}
    \begin{equation}\label{eq:isoperimetric-new}
        \left|\int_{D_R} V(x) \cdot \nu(x) \; \mathrm{d}x\right| \leq C|L_R|^{\frac{d-1}{d-2}} + \varepsilon C R^d.
    \end{equation}
    As long as $\varepsilon \leq C^{-1}R_1^{-1}$, where $R_1$ is defined in the proof of Theorem~\ref{thm:maindivfree}, the extra error term is at most $C^{-1}R^{d-1}$, and therefore does not affect any of the calculations.
\end{proof}

We conclude the section by showing that, even without assuming incompressibility, the reachable set at time $t$ grows proportionally to $t^d$. This lemma plays a key role in ensuring that homogenization occurs in the sense of uniform convergence, by showing that no traps can arise where the reachable set stays bounded for a long time.

\begin{prop}\label{lem:reachablesetgrows}
    There is some $\varepsilon = \varepsilon(d) > 0$ such that, if $|\div V| \leq \varepsilon$ almost surely, then \[ |\mathcal{R}^-_t(x_0)| \geq \frac{t^d}{2d} \] for every $x_0 \in \R^d$ and $t \geq 0$ almost surely.
\end{prop}
\begin{proof}
    Assume for simplicity that $x_0 = 0$ and let $K := t(1+\|V\|_{L^\infty})$. Then $\mathcal{R}^-_t(0) \subseteq Q_K$.

    Since $|\mathcal{R}^-_t(0)| \geq \frac12|B_t|$ for sufficiently small $t \geq 0$, it suffices to show that \[ \partial_t |\mathcal{R}^-_t(0)| = \int_{\partial \mathcal{R}^-_t(0)} {(1 + V(x) \cdot \nu(x))}_+ \; \mathrm{d}x \geq \frac12|\partial \mathcal{R}^-_t(0)|, \] where $\nu$ denotes the outward unit normal to $\mathcal{R}^-_t(0)$.

    By the divergence theorem, \[ \int_{\partial \mathcal{R}^-_t(0)} (1 + V(x) \cdot \nu(x)) \; \mathrm{d}x = |\partial \mathcal{R}^-_t(0)| + \int_{\mathcal{R}^-_t(0)} \div V(x) \; \mathrm{d}x. \]

    To get rid of small parts of the boundary of the reachable set, we define its discretized version by
    \[
        E := \bigcup \left\{x + \overline{Q_{1/2}} \: : \: x \in \Z^d \text{ and } |\mathcal{R}^-_t(0) \cap (x + Q_{1/2})| \geq \frac12\right\}.
    \]
    We will estimate the divergence term by integrating over $E$ instead and using the unit range of dependence. First, we bound the symmetric difference by \[ \left|(E \setminus \mathcal{R}^-_t(0)) \cup (\mathcal{R}^-_t(0) \setminus E)\right| \leq C|\partial \mathcal{R}^-_t(0)|, \] by the isoperimetric inequality applied in each integer-centered unit cube.
    The bound on $\div V$ then implies
    \begin{equation}\label{eq:Eisclose}
        \left|\int_{E} \div V(x) \; \mathrm{d}x - \int_{\mathcal{R}^-_t(0)} \div V(x) \; \mathrm{d}x\right| \leq \varepsilon C |\partial \mathcal{R}^-_t(0)|.
    \end{equation}
    On the other hand, the isoperimetric inequality also yields
    \begin{equation}\label{eq:Eboundarysmall}
        |\partial E| \leq C|\partial \mathcal{R}^-_t(0)|.
    \end{equation}
    We claim that $\int_E \div V(x) \; \mathrm{d}x \leq \varepsilon C |\partial E|$ almost surely. Indeed, if this is true, then since there are only countably many possible values for $E$, the inequality holds for all possible $E$ almost surely. We finish by combining the claim with~\eqref{eq:Eisclose} and~\eqref{eq:Eboundarysmall} to conclude that \[ \int_E \div V(x) \; \mathrm{d}x \leq \varepsilon C |\partial \mathcal{R}^-_t(0), \] so choosing $\varepsilon := \frac12 C^{-1}$ allows us to conclude.

    It remains to prove the claim.
    Let \[ D := \{x \in E \mid \dist(x, \partial E) \leq 1\}. \] Then $|D| \leq C|\partial E|$ (as before, we abuse notation by using $|\cdot|$ to denote the $d$-dimensional measure on the left and $(d-1)$-dimensional measure on the right-hand side), since every integer-centered unit cube in $E$ which intersects $\partial E$ is adjacent to an integer-centered unit cube in $E$ which has at least one of its faces contained in $\partial E$.

    By the divergence theorem, \[ \int_E \div V(x) \; \mathrm{d}x = \int_{\partial E} V(x) \cdot \nu(x) \; \mathrm{d}x, \] where $\nu$ denotes the outward unit normal to $E$. The integral on the right-hand side depends only on $V$ restricted to $\partial E$, and is therefore independent from the random variable \[ \int_{E \setminus D} \div V(x) \; \mathrm{d}x. \] However, we have \[ \int_E \div V(x) \; \mathrm{d}x = \int_D \div V(x) \; \mathrm{d}x + \int_{E \setminus D} \div V(x) \; \mathrm{d}x \leq \varepsilon C |\partial E| + \int_{E \setminus D} \div V(x) \; \mathrm{d}x. \] Taking the conditional expectation with respect to $\mathcal{G}(E \setminus D)$ and using independence yields the claim.
\end{proof}

\section{Global waiting time estimates}
In this section, we use percolation techniques to transform our local waiting time estimates into global bounds.

\subsection{Percolation estimates}
First, we prove a few well-known estimates from supercritical percolation theory, but with a finite range of dependence assumption. We follow the path in~\cite{me}, carefully keeping track of the dependence of constants on the range of dependence with the help of our local waiting time estimates. The proofs are nearly identical to those in~\cite{me}, but we include them here for sake of completion.

Let $d \geq 2$ and let $G \colon \Z^d \to \{0, 1\}$ be a random function on $\Z^d$ which has finite range $C_{\text{dep}} > 0$ of dependence, which means that the $\sigma$-algebras induced by the values of $G$ on sets that are Euclidean distance at least $C_{\text{dep}}$ apart are independent. We assume that $G$ is $\Z^d$-translation invariant, i.e., $G(\cdot+v)$ has the same distribution as $G(\cdot)$ for all $v \in \Z^d$. The function $G$ models site percolation, where a site $x$ is open if $G(x) = 1$ and closed otherwise.

We put edges on $\Z^d$ between nearest neighbors in the $\ell^\infty$ metric. We write $\dist(\cdot, \cdot)$ to indicate the graph distance, and refer to maximal connected components on which $G$ is constant as \textit{clusters}. Clusters composed of open (resp.\ closed) sites are called open (resp.\ closed) clusters. Let $p := \P[G(0) = 1]$ be the probability that a site is open (which is the same for every site, by $\Z^d$-translation invariance). We write $\mathcal{Q}_R(x) \subseteq \Z^d$ to denote the axis-aligned cube of side length $2R$ centered at $x$.
\begin{lem}\label{closed-clusters-small}
    Let $S \subseteq \Z^d$ be a finite set. Define the closed sites connected to $S$ by \[ \cl(S) := \left\{x \in \Z^d \mid \text{there is a path of closed sites from $x$ to a site in $S$}\right\}. \] For any $\varepsilon > 0$ there is $C = C(d) > 0$ such that if \[ p \geq 1-\exp(-C \varepsilon^{-1} C_\text{dep}^d), \] then \[ \P\left[|\cl(S)| > \varepsilon|S| + \delta\right] \leq C\exp(-C^{-1}C_\text{dep}^{-d}\delta). \]
\end{lem}
Note in particular that $x \in \cl(S)$ implies that $x$ is a closed site.
\begin{proof}
    Let $T \subseteq \Z^d$ be a finite set of $n$ vertices. If $T \subseteq \cl(S)$, then every site in $T$ is closed and every cluster in $T$ contains a point in $S$. For fixed $n$, the number of sets $T$ which satisfy the latter condition is at most ${(2d+2)}^{|S|+2n}$ (we can encode a spanning tree of $T$ with an alphabet of $2d+2$ letters). For a fixed set $T$, we see that \[ \P[\text{every site in $T$ is closed}] \leq {(1-p)}^{\left\lfloor n/{(2C_\text{dep})}^d\right\rfloor}, \] since we can choose at least $\left\lfloor n/{(2C_\text{dep})}^d\right\rfloor$ sites in $T$ which are far enough to be independent. From the union bound, we have \[\P[\text{there is a set $T \subseteq \cl(S)$ with $n$ vertices}] \leq {(2d+2)}^{|S|+2n}{(1-p)}^{\left\lfloor n/{(2C_\text{dep})}^d\right\rfloor}. \] Now let $n := \left\lceil\varepsilon|S| + \delta\right\rceil$ and choose $(1-p)$ small enough so that \[{(1-p)}^{\varepsilon/{(2C_\text{dep})}^d}{(2d+2)}^{1+2\varepsilon} < 1\] and \[{(1-p)}^{1/{(2C_\text{dep})}^d}{(2d+2)}^2 < 1\] and the claim follows.
\end{proof}
The next lemma has nothing to do with the percolation environment; it's simply a property of the graph structure of $\Z^d$. It follows from a topological property of $\R^d$ known as unicoherence (see Kuratowski~\cite{Kuratowski} or Dugundji~\cite{Dugundji}). In order to state the lemma, we need to define the boundary of a subset of $E \subseteq \Z^d$. Because $\Z^d$ is discrete, there are two choices for our definition.
\begin{defn}
    The inner (resp.\ outer) boundary of $E$, denoted $\partial^-E$ (resp. $\partial^+E$), is the set \[ \partial^-E := \{x \in E \mid \dist(x, \Z^d \setminus E) = 1\} \qquad \text{(resp.\ $\partial^+E := \{x \in \Z^d \setminus E \mid \dist(x, E) = 1\}$)}. \]
\end{defn}

\begin{lem}\label{unicoherence}
    Let $\mathcal{Q}_R$ be any cube of side length $2R$ and let $\mathfrak{C} \subseteq \mathcal{Q}_R$ be a connected set. Let $\mathfrak{D} \subseteq \mathcal{Q}_R \setminus \mathfrak{C}$ be a connected component of $\mathcal{Q}_R \setminus \mathfrak{C}$. Then the inner (resp.\ outer) boundary of $\mathfrak{D}$ is connected.
\end{lem}
\begin{proof}
    This is part (i) of Lemma 2.1 from Deuschel\textendash{}Pisztora~\cite{DeusPisz}. The proof is a standard application of Urysohn's lemma.
\end{proof}

The next lemma shows that, with high probability, there is a large open cluster which is near every site.
\begin{lem}\label{open-cluster-big}
    Let $n, R > 0$ and consider $\mathcal{Q}_R$, a cube of side length $2R$. Let $E_n$ be the event that there exists an open cluster $\mathfrak{C} \subseteq \mathcal{Q}_{R+n}$ such that every connected component of $\mathcal{Q}_{R+n} \setminus \mathfrak{C}$ which intersects $\mathcal{Q}_R$ is of size at most $n$. Then there is a constant $C = C(d) > 0$ such that if \[ p \geq 1-\exp(-C \varepsilon^{-1} C_\text{dep}^d), \] then \[\P[E_n] \geq 1-CR^d\exp(-C^{-1}C_\text{dep}^{-d}n^{(d-1)/d}).\]
\end{lem}
\begin{proof}
    Fix $n, R > 0$ as in the statement. Work in the event that every closed cluster in $\mathcal{Q}_{R+n}$ has size less than $C^{-1}n^{(d-1)/d}$, where $C = C(d)$ is related to the isoperimetric constant (it will be chosen by the end of the proof). By Lemma~\ref{closed-clusters-small} (applied to each site in $\mathcal{Q}_{R+n}$ individually, with (say) $\varepsilon = 1$), this event has probability at least $1-CR^d\exp(-C^{-1} C_\text{dep}^{-d}n^{(d-1)/d})$.

    Let $\mathfrak{C}$ be the largest open cluster (breaking ties arbitrarily) in $\mathcal{Q}_{R+n}$. As long as $C \geq 1$, it follows from the isoperimetric inequality that there is an open path between opposite faces of $\mathcal{Q}_{R+n}$, so $|\mathfrak{C}| \geq 2(R+n)+1 > n$.

    Let $\mathfrak{D}$ be any connected component of $\mathcal{Q}_{R+n} \setminus \mathfrak{C}$ which intersects $\mathcal{Q}_R$. The inner boundary of $\mathfrak{D}$ is composed of two kinds of sites: (i) those bordering $\mathfrak{C}$ and (ii) those in the inner boundary of $\mathcal{Q}_{R+n}$. The sites of type (i) are all closed (by definition of $\mathfrak{C}$). We claim that there are no sites of type (ii). Indeed, if there was a site of type (ii) then we could follow the inner boundary of $\mathfrak{D}$ (it's connected by Lemma~\ref{unicoherence}) from the inner boundary of $\mathcal{Q}_{R+n}$ to a site in $\mathcal{Q}_R$, which would yield a path of length more than $n \geq n^{(d-1)/d}$ of type (i) (and hence closed) sites, contradicting our assumption.

    Since the inner boundary of $\mathfrak{D}$ is connected and composed entirely of closed sites, it has size less than $C^{-1}n^{(d-1)/d}$. The isoperimetric inequality then shows that either $\mathfrak{D}$ or $\mathcal{Q}_{R+n} \setminus \mathfrak{D}$ has size at most $n$. Since $\mathfrak{C} \subseteq \mathcal{Q}_{R+n} \setminus \mathfrak{D}$ and $|\mathfrak{C}| > n$, it follows that $|\mathfrak{D}| \leq n$ as desired.
\end{proof}

\subsection{A global bound}
Next, we show that the region where the waiting time is small contains a supercritical percolation cluster.

\begin{lem}\label{it-is-percolation}
    There is a constant $C = C(d) > 0$ such that for each $0 < p < 1$, if \[ |\div V| \leq C^{-1}{(\|V\|_{C^{0,1}}+1)}^{-C} {\left(\log {(1-p)}^{-1}\right)}^{-1/(d-1)} \] almost surely, then the function $G \colon \Z^d \to \{0, 1\}$, defined by \[ G(v) = \begin{cases} 1 & \quad \theta(x, y) \leq C{(\|V\|_{C^{0,1}}+1)}^{C} \log {(1-p)}^{-1} \; \text{for all $x, y \in B_{\sqrt{d}}(v)$}\\ 0 & \quad \text{otherwise,} \end{cases} \] is $Z^d$-translation invariant with finite range of dependence \[ C_\text{dep} \leq C {(\|V\|_{C^{0,1}}+1)}^{C} \log {(1-p)}^{-1} \] and $\P[G(0) = 1] \geq p$.
\end{lem}
In other words, $G$ is an environment in which all of our percolation estimates apply.
\begin{proof}
    By Proposition~\ref{prop:mainsmalldiv}, there is $C > 0$ such that $\P[G(0) = 1] \geq p$. The $\Z^d$-translation invariance of $G$ follows from that of $\P$. Finite range of dependence follows from the fact that the value of $G(v)$ depends only on controlled paths starting in $B_{\sqrt{d}}(v)$ which run for time at most $C \|V\|_{C^{0,1}}^C \log {(1-p)}^{-1}$. The bound on $C_\text{dep}$ follows, noting that the top speed of a path is $1+\|V\|_{L^\infty}$.
\end{proof}

\begin{defn}
    To translate between $\Z^d$ and $\R^d$, for each set $E \subseteq \Z^d$ we introduce the ``solidification'' \[\sigma(E) := E+{\left[-\frac12, \frac12\right]}^d. \]
\end{defn}
We are now in a position to prove our global controllability estimate. We refer to clusters as before, using the same notion of adjacency.
\begin{thm}\label{cont-estimate}
    There is a constant $C = C(d) > 0$ such that, if \[ |\div V| \leq C^{-1}{(\|V\|_{C^{0,1}}+1)}^{-C}, \] then for each $R \geq 1$, the ``extra waiting time'' \[ \mathcal{E}(R) := \sup_{x, y \in Q_{R}} \frac{\theta(x, y) - C(1 + |x-y|)}{C{(\|V\|_{C^{0,1}}+1)}^{C}} \] satisfies \[ \P[\mathcal{E}(R) > n] \leq CR^d\exp(-C^{-1} {(\|V\|_{C^{0,1}}+1)}^{-C}n). \]
\end{thm}
\begin{proof}
    We partition $\R^d$ into cubes of side length $1$, centered at points in $\Z^d$. Fix $p := 1-\exp(-C C_\text{dep}^d)$, where $C_\text{dep}$ is given by Lemma~\ref{it-is-percolation}. By Lemma~\ref{it-is-percolation}, if \[ |\div V| \leq C^{-1}{(\|V\|_{C^{0,1}}+1)}^{-C}, \] then $\P[G(v) = 1] \geq p$. For $v \in \Z^d$, we say that the site $v$ is open if $G(v) = 1$ and closed otherwise. We say that a point $x \in \R^d$ lies near an open site if there is some $v \in \Z^d$ such that $x \in \sigma(\{v\})$. Let $S = \lceil R \rceil$ and $x, y \in Q_{S}$. For convenience, we will prove that $\P[\mathcal{E}(R) > Cn] \leq CR^d\exp(-C^{-1}{(\|V\|_{C^{0,1}}+1)}^{-C}n)$; this easily implies the original claim by changing $C$ to $C^2$.

    \textit{Step 1.} We show that we can assume without loss of generality that $x$ and $y$ lie in the same open cluster, by which we mean that there is an open cluster $\mathfrak{C}$ such that $x, y \in \sigma(\mathfrak{C})$. Indeed, suppose they lie in different open clusters. Then by Lemma~\ref{open-cluster-big}, we have an open cluster $\mathfrak{C} \subseteq \mathcal{Q}_{S+n^d}$, which depends on the environment, such that \[ \P\left[\text{every connected $\mathfrak{D} \subseteq \mathcal{Q}_{S+n^d} \setminus \mathfrak{C}$ which intersects $\mathcal{Q}_S$ satisfies $|\mathfrak{D}| \leq n^d$}\right] \geq 1-CS^d\exp(-C^{-1}C_\text{dep}^{-d}n^{d-1}).\] Working in this event, let $\mathfrak{D}$ be the connected component of $\mathcal{Q}_{S+n^d} \setminus \mathfrak{C}$ whose solidification $\sigma(\mathfrak{D})$ contains $x$. Then Lemma~\ref{lem:reachablesetgrows}, along with the fact that $\mathcal{R}_t^-(x)$ is connected, shows that $\mathcal{R}_t^-(x) \setminus \sigma(\mathfrak{D}) \neq \emptyset$ if $|t|^d \geq 2d n^d$ and $t > 0$. Since the controlled paths are continuous, there is some $t \leq n {(2d)}^{1/d} + 1$ such that there is some $z \in \mathcal{R}_t^-(x) \cap \sigma(\mathfrak{C})$. Since the loss in probability and travel time can be controlled by enlarging $C$, we may as well assume that $x$ was $z$ to begin with. Repeating the same argument for $y$ (except running time backwards, so $t < 0$ and we use $\mathcal{R}_t^+(y)$ instead) shows that we may assume that $x$ and $y$ lie in the same open cluster, $\mathfrak{C}$.

    \textit{Step 2.} To show that \[ \theta(x, y) - C|x-y| \leq C{(\|V\|_{C^{0,1}}+1)}^{C}n, \] we will build a ``skeleton'' of points $x = x_0, x_1, \dots, x_k = y$ which all lie near open sites and satisfy $|x_{i+1}-x_i| \leq \sqrt{d}$ for each $0 \leq i < k$. Then, by connecting the points with paths given by Lemma~\ref{it-is-percolation}, we can build a controlled path of length at most $C{(\|V\|_{C^{0,1}}+1)}^{C}k$ which follows the skeleton. It remains to show that we can build such a skeleton with $k \leq C(|x-y| + n)$. Let \[A := \{v \in \Z^d \mid \sigma(\{v\}) \cap \overline{xy} \neq \emptyset\}\] be the set of centers of cubes which intersect the line segment connecting $x$ and $y$. Note that $|A| \leq 2^d(1+|x-y|)$. By Lemma~\ref{closed-clusters-small}, we can choose $\varepsilon = 1$ to work in the event that $\cl(A) \leq |A| + n$.

    Our strategy is to go from $x$ to $y$ in a straight line, taking necessary detours around closed clusters. We use Lemma~\ref{closed-clusters-small} to bound the total length of our detour.

    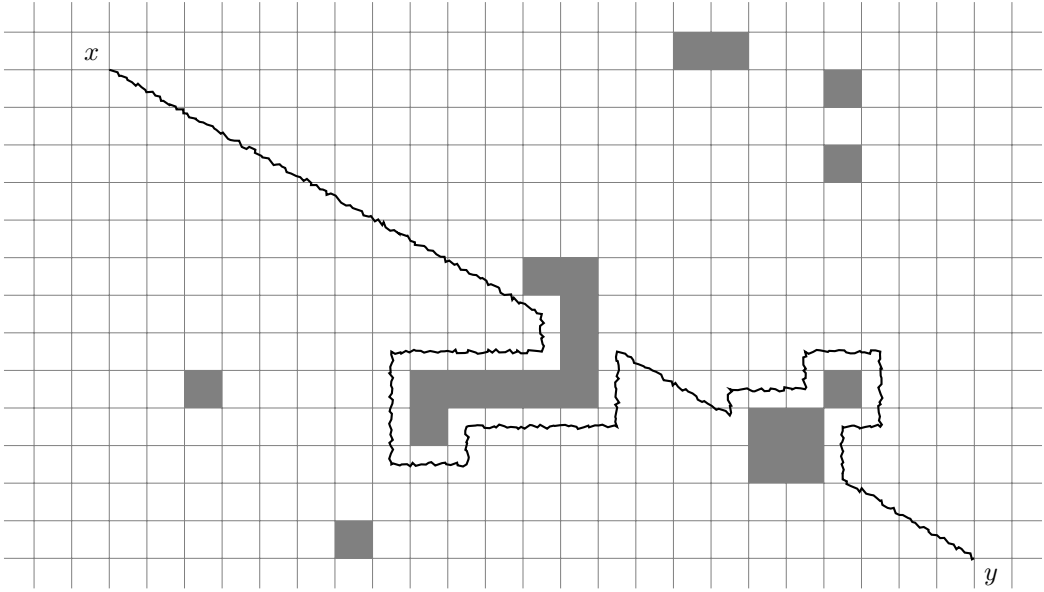
\begin{figure}
        \centering
        \begin{tikzpicture}
    \usetikzlibrary{decorations.pathmorphing}
    \draw[step=0.5cm, gray, very thin] (-6.9, -3.9) grid (6.9, 3.9);
    \fill[gray] (0, 0) rectangle (0.5, 0.5);
    \fill[gray] (0.5, 0) rectangle (1, 0.5);
    \fill[gray] (2, 3) rectangle (2.5, 3.5);
    \fill[gray] (2.5, 3) rectangle (3, 3.5);
    \fill[gray] (-1.5, -2) rectangle (-1, -1.5);
    \fill[gray] (0.5, -0.5) rectangle (1, 0);
    \fill[gray] (0.5, -1) rectangle (1, -0.5);
    \fill[gray] (0.5, -1.5) rectangle (1, -1);
    \fill[gray] (0, -1.5) rectangle (0.5, -1);
    \fill[gray] (-0.5, -1.5) rectangle (0, -1);
    \draw[black, thick, decorate, decoration={random steps, segment length=0.05cm, amplitude=0.03cm}] (-5.5, 3) node[anchor=south east] {$x$} -- (0.25, -0.25) -- (0.25, -0.75) -- (-1.75, -0.75) -- (-1.75, -2.25) -- (-0.75, -2.25) -- (-0.75, -1.75) -- (1.25, -1.75) -- (1.25, -0.75) -- (2.75, -1.6) -- (2.75, -1.25) -- (3.75, -1.25) -- (3.75, -0.75) -- (4.75, -0.75) -- (4.75, -1.75) -- (4.25, -1.75) -- (4.25, -2.45) -- (6, -3.5) node[anchor=north west] {$y$};
    \fill[gray] (-1, -1.5) rectangle (-0.5, -1);
    \fill[gray] (-1.5, -1.5) rectangle (-1, -1);
    \fill[gray] (-2.5, -3.5) rectangle (-2, -3);
    \fill[gray] (-4.5, -1.5) rectangle (-4, -1);
    \fill[gray] (4, -1.5) rectangle (4.5, -1);
    \fill[gray] (3.5, -2) rectangle (4, -1.5);
    \fill[gray] (3, -2) rectangle (3.5, -1.5);
    \fill[gray] (3.5, -2.5) rectangle (4, -2);
    \fill[gray] (3, -2.5) rectangle (3.5, -2);
    \fill[gray] (4, 2.5) rectangle (4.5, 3);
    \fill[gray] (4, 1.5) rectangle (4.5, 2);
\end{tikzpicture}
        \caption{An example of our controlled path from $x$ to $y$; cubes corresponding to closed sites are shaded}
    \end{figure}

    We can build the skeleton iteratively. Start with $x_0 = x$, and assume we've built the skeleton up to $x_i$, for $i \geq 0$. We maintain the invariant that, at the end of each step, $x_i$ lies on the line segment $\overline{xy}$, and it is closer to $y$ than any other $x_j$ which lies on $\overline{xy}$ with $0 \leq j < i$. There are three cases.
    \begin{enumerate}
        \item If $|y-x_i| \leq \sqrt{d}$, define $x_{i+1} := y$ and finish.
        \item If $z := x_i + \sqrt{d}\frac{y-x}{|y-x|}$ lies near an open site, define $x_{i+1} := z$ and continue to the next step.
        \item Otherwise, let $\tilde{x}$ be the point on $\overline{x_i z}$ which lies near an open site and is closest to $z$. Then $\tilde{x}$ also lies near a closed site, which is part of some connected component $\mathfrak{F}$ of $\mathcal{Q}_S \setminus \mathfrak{C}$. By Lemma~\ref{unicoherence}, the outer boundary of $\mathfrak{F}$ is connected. Since $\tilde{x} \in \overline{xy}$, we see that $\partial^- \mathfrak{F} \subseteq \cl(A)$. Besides, since every vertex in $\Z^d$ has degree $3^d-1$, we have $|\partial^+ \mathfrak{F}| \leq 3^d |\partial^- \mathfrak{F}|$. So, let $p_1, \dots, p_\ell \in \Z^d$ be a path along the outer boundary of $\mathfrak{F}$, where $\sigma(\{p_1\}) \ni \tilde{x}$ and $p_\ell \in A$ is a point on the outer boundary of $\mathfrak{F}$ which maximizes $p_\ell \cdot (y-x)$. Finally, we extend our path by setting $x_{i+1} := \tilde{x}$ and $x_{i+j+1} := p_j$ for each $1 \leq j \leq \ell$, and set $x_{i+\ell+2}$ to be a point on the line segment $\overline{xy}$ which lies in $\sigma(\{p_\ell\})$.
    \end{enumerate}

    It remains to analyze the length of this skeleton by looking at each of the three cases. The first case happens at most once, so it can be ignored. The second case reduces the distance $|x_i-y|$ by $\sqrt{d}$ and the third case does not increase the distance $|x_i-y|$, so there can be at most $\frac{|x-y|}{\sqrt{d}}$ points in the skeleton coming from the second case. The third case adds $\ell+2$ points, where $\ell \leq |\partial^+\mathfrak{F}|$. Since we finish an instance of the third case at a point as close to $y$ as possible on the segment $\overline{xy}$, we never witness the same cluster $\mathfrak{F}$ twice in different instances of the third case. Therefore the third case adds at most \[ C|\cl(A)| \leq C(|A| + n) \leq C(|x-y| + n) \] points to our skeleton.
\end{proof}

\section{Random fluctuations in first-passage time}
Next, we consider how much $\theta(0, y)$ deviates from its expectation. Our proof will follow roughly the same path as the proof of Proposition 4.1 from Armstrong\textendash{}Cardaliaguet\textendash{}Souganidis~\cite{ArmsCardSoug}, with some modifications which are made possible by the controllability estimate. As in the previous section, the proofs are nearly identical to those in~\cite{me}, with a bit of care taken to keep track of the dependence of the constants on $V$.

To get started, we introduce a ``guaranteed'' version of first passage time. For any $\rho > 0$, define the $\rho$-guaranteed reachable set recursively by \[ \mathcal{R}_t^\rho(x) := \begin{cases} \mathcal{R}_t^-(x) & \quad \text{if $t < \rho$}\\ \mathcal{R}_\rho^-(\mathcal{R}_{t-\rho}^\rho(x)) \cup (\mathcal{R}_{t-\rho}^\rho(x) + \overline{B_1}) & \quad \text{otherwise.} \end{cases} \] The $\rho$-guaranteed reachable set is similar to the reachable set, except that we enforce expansion at a rate of at least $1/\rho$ in a certain discrete sense. We similarly define the $\rho$-guaranteed first passage time \[ \theta^\rho(x, y) = \min \{t \geq 0 \mid y \in \mathcal{R}_t^\rho(x) \}. \]

    Note that the $\rho$-guaranteed first passage time coincides with the usual first passage time if we have sufficient control on the extra waiting time $\mathcal{E}$ (from Theorem~\ref{cont-estimate}) in a suitable domain.

    Fix some $y \in \R^d$ and define the random variable ${\{Z^\rho_t\}}_{t \geq 0}$ by \[ Z^\rho_t := \E\left[\theta^\rho(0, y) \mid \mathcal{F}_t\right], \] where $\mathcal{F}_t$ is the $\sigma$-algebra generated by the environment $V(x)$ restricted to the $\rho$-guaranteed reachable set $\mathcal{R}_t^\rho(0)$. In other words, $\mathcal{F}_t$ is the smallest $\sigma$-algebra so that the functions $V(x)\mathds{1}_{x \in \mathcal{R}_t^\rho(0)}$ are $\mathcal{F}_t$-measurable for every $x \in \R^d$. Since $\mathcal{R}_t^\rho(0)$ are increasing sets, ${\{\mathcal{F}_t\}}_{t \geq 0}$ is a filtration, so ${\{Z^\rho_t\}}_{t \geq 0}$ is a martingale.

    We first show that $Z^\rho_t$ depends mostly on the shape of $\mathcal{R}_t^\rho(0)$, without regard for the values of $V$ inside $\mathcal{R}_t^\rho(0)$. In order to condition on the approximate shape of the reachable set, for any $E \subseteq \R^d$ we introduce the discretization \[ \disc(E) := \{z \in d^{-1/2}\Z^d \mid B(z, 1) \cap E \neq \emptyset\}. \]
    \begin{lem}\label{fluc-helper-1}
        For any $t \geq 0$, we have \[ \left|\max(Z^\rho_t, t) - f(t, \disc(\mathcal{R}_t^\rho(0)))\right| \leq 3\rho, \] where we define $f(t, S)$, for any $t \geq 0$ and any finite set $S \subseteq d^{-1/2}\Z^d$, by \[ f(t, S) := t + \E\left[\theta^\rho(S, y)\right]. \]
    \end{lem}
    \begin{proof}
        Fix some $t \geq 0$. Using the speed limit $1+\|V\|_{L^\infty}$ for controlled paths, we see that \[ \mathcal{R}_t^\rho(0) \subseteq B(0, 1+\lceil \|V\|_{L^\infty} + 1 + \rho^{-1} \rceil t) \] almost surely. Define the set of possible discretized reachable sets at time $t$ by
        \begin{equation}\label{C_t-definition}
            C_t := \left\{S \subseteq d^{-1/2}\Z^d \cap B\left(0, \lceil \|V\|_{L^\infty} + 1 + \rho^{-1} \rceil t\right)\right\}.
        \end{equation}
        For any $S \in C_t$, the event that $\disc(\mathcal{R}^\rho_t(0)) = S$ is $\mathcal{F}_t$-measurable, so we have \[ \left|\max\left(Z^\rho_t, t\right) - f\left(t, \disc(\mathcal{R}_t^\rho(0))\right)\right| = \sum_{S \in C_t} \max\left(\left|\E\left[\left(\theta^\rho(0, y)-f(t, S)\right)\mathds{1}_{\disc(\mathcal{R}^\rho_t(0)) = S} \mid \mathcal{F}_t\right]\right|, \left|t-f(t, S)\right|\right). \]
        Fix some $S \in C_t$. If $B(y, 2) \cap S \neq \emptyset$, then $y \in \mathcal{R}^\rho_{2\rho}(S)$, so $\theta^\rho(0, y) \leq t+3\rho$ and $t \leq f(t, S) \leq t+3\rho$ and the conclusion holds.

        Otherwise, define the set $E := S + B(0, 2)$. Note that $\mathcal{R}^\rho_t(0) \subseteq E$ and $\dist(\mathcal{R}_t^\rho(0), \partial E) \geq 1$. Using the definition of the $\rho$-guaranteed reachable set,
        \begin{equation}\label{theta-p-inequality}
            \theta^\rho(0, \partial E) + \theta^\rho(\partial E, y) - \rho \leq \theta^\rho(0, y) \leq \theta^\rho(0, \partial E) + \theta^\rho(\partial E, y).
        \end{equation}
        Using the definitions of $S$ and $E$, \[ t \leq \theta^\rho(0, \partial E) \leq t + 3\rho. \] On the other and, the term $\theta^\rho(\partial E, y)$ is $\mathcal{G}(\R^d \setminus E)$-measurable. Taking the conditional expectation of~(\ref{theta-p-inequality}), we have \[ t - \rho + \E[\theta^\rho(\partial E, y)] \leq Z^\rho_t \leq t + 3\rho + \E[\theta(\partial E, y)]. \] To finish, we use the definition of the $\rho$-guaranteed reachable set to find that \[ 0 \leq \theta^\rho(S, y) - \theta^\rho(\partial E, y) \leq 2\rho. \] Combining the previous two displays yields the conclusion of the lemma.
    \end{proof}

    Next, we show that our approximation for $Z_t^\rho$, given by $f(t, \disc(\mathcal{R}_t^\rho(0)))$, has bounded increments.
    \begin{lem}\label{fluc-helper-2}
        Let $t, s \geq 0$. Then \[ \left|f(t, \disc(\mathcal{R}_t^\rho(0))) - f(s, \disc(\mathcal{R}_s^\rho(0)))\right| \leq 2\rho + |t-s|(\|V\|_{L^\infty}\rho + \rho + 2). \]
    \end{lem}
    \begin{proof}
        Without loss of generality, assume $s < t$. Let $C_s$ and $C_t$ be as defined in~(\ref{C_t-definition}). We need to prove that \[ \sum_{A_s \in C_s} \sum_{A_t \in C_t} \left|f(t, A_t) - f(s, A_s)\right|\mathds{1}_{\disc(\mathcal{R}_s^\rho(0)) = A_s}\mathds{1}_{\disc(\mathcal{R}_t^\rho(0)) = A_t} \leq 2\rho + |t-s|(\|V\|_{L^\infty}\rho + \rho + 2). \] So, fix any $A_s \in C_s$ and $A_t \in C_t$ such that \[ \P\left[\disc(\mathcal{R}_s^\rho(0)) = A_s \text{ and } \disc(\mathcal{R}_t^\rho(0)) = A_t\right] > 0.\] The speed limit for controlled paths shows that $A_s \subseteq A_t$ and \[ \dist_H(A_s, A_t) \leq 2 + |t-s|\lceil 1+\|V\|_{L^\infty}+\rho^{-1}\rceil, \] where $\dist_H$ denotes the Hausdorff distance. Using the definition of the $\rho$-guaranteed reachable set, this yields \[ \theta^\rho(A_t, y) \leq \theta^\rho(A_s, y) \leq \theta^\rho(A_t, y) + \rho\left(2+|t-s|\lceil 1+\|V\|_{L^\infty}+\rho^{-1}\rceil\right). \] The conclusion of the lemma follows from the definition of $f$.
    \end{proof}

    Now we put these lemmas together and apply Azuma's inequality to ${\{Z^\rho_t\}}_{t \geq 0}$, choosing $\rho$ carefully to balance competing error terms.
    \begin{prop}\label{prop:random-fluc}
        There is a constant $C = C(d) > 0$ such that, if $y_1, y_2 \in \R^d$, \[ \lambda \geq C{(\|V\|_{C^{0,1}}+1)}^{C}|y_1-y_2|^{1/2}\log^2|y_1-y_2|, \] and \[ |\div V| \leq C^{-1}{(\|V\|_{C^{0,1}}+1)}^{-C}, \] then \[ \P[|\theta(y_1, y_2) - \E[\theta(y_1, y_2)]| > \lambda] \leq C\exp\left(\frac{-C^{-1} {(\|V\|_{C^{0,1}}+1)}^{-C}\lambda^{1/2}}{|y_1-y_2|^{1/4}}\right). \]
    \end{prop}
    \begin{proof}
        We assume, for convenience of notation, that $y_1 = 0$; let $y := y_2$.

        Note that $\theta(0, y) = \theta^\rho(0, y)$ as long as $\theta(u, v) \leq \rho$ if $|u-v| \leq 1$, for all $|u|, |v| \leq \lceil (1+\|V\|_{L^\infty}+\rho^{-1}) \rceil \lceil \rho^{-1}|y| \rceil$ (a ball of this radius contains the $\rho$-guaranteed reachable set at time $\lceil \rho^{-1}|y|\rceil$). By Lemma~\ref{fluc-helper-1} and Lemma~\ref{fluc-helper-2}, the martingale ${\{Z^\rho_t\}}_{t \geq 0}$ has bounded increments of \[ |Z_t - Z_s| \leq (\rho \|V\|_{L^\infty}+\rho+2)|t-s| + 8\rho. \] Also, $Z_t^\rho = \theta(0,y)$ for all $t \geq \rho|y|$. Apply the union bound with Theorem~\ref{cont-estimate} and Azuma's inequality to the sequence ${\{Z^\rho_n\}}_{0 \leq n \leq \lceil \rho^{-1}|y| \rceil}$ to see that
        \begin{align*}
            \P\left[\left|\theta(0, y) - \E[\theta^\rho(0, y)]\right| > \lambda\right] &\leq \P[\theta(0, y) \neq \theta^\rho(0, y)] + \P\left[\left|\theta(0, y) - \E[\theta^\rho(0, y)]\right| > \lambda\right]\\
            &\leq C{\left[(1+\|V\|_{L^\infty}+\rho^{-1})\rho^{-1}|y|\right]}^d\exp\left(-C^{-1} {(\|V\|_{C^{0,1}}+1)}^{-C}\rho\right)\\
            &\qquad + 2\exp\left(\frac{-C^{-1}\lambda^2}{{\left[(\rho L+\rho+2) + 8\rho\right]}^2\left\lceil \rho^{-1}|y|\right\rceil}\right),
        \end{align*}
        as long as $\rho \geq 2C$.
        Choosing $\rho = \lambda^{1/2}|y|^{-1/4}$ yields the bound
        \begin{equation}\label{eq:ok-but-replace}
            \P[|\theta(0, y) - \E[\theta^\rho(0, y)]| > \lambda] \leq C\exp\left(\frac{-C^{-1} {(\|V\|_{C^{0,1}}+1)}^{-C}\lambda^{1/2}}{|y|^{1/4}}\right),
        \end{equation}
        as long as
        \begin{equation}\label{eq:lambdalowerbound}
            \lambda \geq C{(\|V\|_{C^{0,1}}+1)}^{C}|y|^{1/2}\log^2|y|
        \end{equation}
        (we absorb all the polynomials into the exponential by changing the constant appropriately). It remains to replace $\E[\theta^\rho(0, y)]$ in~\eqref{eq:ok-but-replace} with $\E[\theta(0, y)]$. Indeed, we can bound
        \begin{align}
            \E[\theta(0, y)] &= \E\left[\theta^\rho(0, y) \mid E_\rho\right]\cdot\P[E_\rho] + \E\left[\theta(0, y) \mid E_\rho^c\right]\cdot(1-\P[E_\rho])\nonumber\\
            &= \E\left[\theta^\rho(0, y)\right] + O\left(\rho|y|(1-\P[E_\rho]) + {(\rho|y|)}^{d+1}\exp(-C^{-1}{(\|V\|_{C^{0,1}}+1)}^{-C}\rho)\right)\label{eq:almost-done-bound},
        \end{align}
        where $E_\rho$ denotes the event that $\theta(0, y) = \theta^\rho(0, y)$ and the last part of the last line comes from the controllability bound in Theorem~\ref{cont-estimate}. The lower bound~\eqref{eq:lambdalowerbound} on $\lambda$ (and hence on $\rho$) ensures that the error term in~\eqref{eq:almost-done-bound} is at most $\frac{1}{2}\lambda$ and can therefore be absorbed into the constant.
    \end{proof}

    \section{Nonrandom scaling bias}
    In this section, we use the bounds on random fluctuations of $\theta$ to bound the difference between $\E[\theta(0, y)]$ and $\lim_{\varepsilon \to 0^+} \varepsilon \E[\theta(0, \varepsilon^{-1}y)]$, which we refer to as the \textit{nonrandom scaling bias}. We follow a similar argument as in Alexander~\cite{Alexander90}, who proved an analogous result for Bernoulli percolation in two dimensions. Happily, the argument goes through in any dimension with the help of the Hobby\textendash{}Rice theorem~\cite{HobbyRice}, a version of which we quote below. We include their proof, because it is short and beautiful.

    \begin{thm}[Hobby\textendash{}Rice]\label{thm:hobbyrice}
        Let $\gamma \colon [0, 1] \to \R^d$ be continuous. Then there is a partition \[ 0 = t_0 < t_1 < \cdots < t_{d+1} = 1, \] along with signs \[ \delta_1, \dots, \delta_{d+1} \in \{-1, +1\}, \] such that \[ \sum_{k=1}^{d+1} \delta_k(\gamma(t_k)-\gamma(t_{k-1})) = 0. \]
    \end{thm}
    \begin{proof}
        We parameterize signed partitions by points on the $d$-sphere as follows. Given a point $x \in S^d$, we define associated signs by $\delta^x_k = \sgn(x_k)$ and define ${\{t_k\}}_k$ to be the unique partition of $[0, 1]$ such that $t^x_{k}-t^x_{k-1} = x_k^2$. As such, we define the map $f \colon S^d \to \R^d$ by \[ {f(x)} := \sum_{k=1}^{d+1}\delta^x_k(\gamma(t^x_k)-\gamma(t^x_{k-1})). \] By the Borsuk\textendash{}Ulam theorem, there is some $x \in S^d$ such that $f(x) = f(-x)$. However, $f$ is odd, so $f(x) = 0$, which proves the claim.
    \end{proof}

    We now bound the nonrandom scaling bias. Given a function $f \colon \R^d \to \R$, we define the large-scale limit $\overline{f} \colon \R^d \to \R$ by \[ \overline{f}(x) := \lim_{\varepsilon \to 0^+} \varepsilon f\left(\varepsilon^{-1}x\right). \]
    \begin{prop}\label{prop:scalingbias}
        Assume that the law of $V$ is $\Z^d$-translation invariant and that \[ |\div V| \leq C^{-1}{(\|V\|_{C^{0,1}}+1)}^{-C}. \] Let $f(x) := \E[\theta(0, x)]$. Then \[ |f(x) - \overline{f}(x)| \leq C{(\|V\|_{C^{0,1}}+1)}^{C}|x|^{1/2}\log^2|x| \] for all $|x| \geq 1$.
    \end{prop}
    \begin{proof}
        First, note that translation invariance and the controllability bound in Theorem~\ref{cont-estimate} implies that $f$ is subadditive up to a constant, that is, \[ f(x+y) \leq f(x) + f(y) + C{(\|V\|_{C^{0,1}}+1)}^{C} \] for all $x, y \in \R^d$, so it follows immediately that \[ f(y) \geq \overline{f}(y) +  C{(\|V\|_{C^{0,1}}+1)}^{C}. \] Our goal is to show that $f$ is superadditive up to some small error, after which we apply an argument similar to that in Fekete's lemma to bound the difference between $f$ and its large-scale limit.

        Fix any $y \in \R^d$. By Proposition~\ref{prop:random-fluc}, Theorem~\ref{cont-estimate}, and the union bound, the event that
        \begin{equation}\label{eq:closetoexpectation}
            |\theta(v, w) - \E[\theta(v, w)]| \leq C{(\|V\|_{C^{0,1}}+1)}^{C}|y|^{1/2}\log^2|y|
        \end{equation}for all
        \[ |v|, |w| \leq C{(\|V\|_{C^{0,1}}+1)}^{C}|y| \] has positive probability. By translation invariance, this implies that
        \begin{equation}\label{eq:translateinterval}
            |\theta(w, x) - \theta(y, z)| \leq C{(\|V\|_{C^{0,1}}+1)}^{C}|y|^{1/2}\log^2|y|
        \end{equation}
        whenever $|(x-w)-(z-y)| \leq C$ and \[ |w|, |x|, |y|, |z| \leq C{(\|V\|_{C^{0,1}}+1)}^{C}|y|. \]

        In an instance of this event, let $\gamma \colon [0, \theta(0, y)] \to \R^d$ be a controlled path from $0$ to $y$. Applying Theorem~\ref{thm:hobbyrice} to $\gamma$, we conclude that there are points \[ 0 \leq s_1 < t_1 \leq s_2 < t_2 \leq \cdots \leq s_\ell < t_\ell \leq 1, \] where $\ell \leq \frac{d+1}{2}$, such that \[ \sum_{k=1}^\ell \gamma(t_k) - \gamma(s_k) = \frac12 y. \] Applying~\eqref{eq:closetoexpectation} and~\eqref{eq:translateinterval}, we conclude that
        \begin{align*}
            2f\left(\frac12 y\right) &\leq C{(\|V\|_{C^{0,1}}+1)}^{C}|y|^{1/2}\log^2|y| + \theta\left(0, \frac12 y\right) + \theta\left(\frac12 y, y\right)\\
            &\leq C{(\|V\|_{C^{0,1}}+1)}^{C}|y|^{1/2}\log^2|y| + \left(\sum_{k=1}^\ell \theta(\gamma(s_k), \gamma(t_k))\right)\\
            &\qquad\qquad + \left(\theta(0, \gamma(s_1)) + \theta(\gamma(t_1), y) + \sum_{k=2}^\ell \theta(\gamma(t_{k-1}), \gamma(s_k))\right)\\
            &\leq C{(\|V\|_{C^{0,1}}+1)}^{C}|y|^{1/2}\log^2|y| + \theta(0, y)\\
            &\leq C{(\|V\|_{C^{0,1}}+1)}^{C}|y|^{1/2}\log^2|y| + f(y).
        \end{align*}
        It follows by induction that \[ 2^n f(y) \leq f\left(2^n y\right) + \sum_{k=0}^{n-1} 2^{n-1-k}C{(\|V\|_{C^{0,1}}+1)}^{C}|2^k y|^{1/2}\log^2|2^k y|. \]
        Dividing by $2^n$ on both sides and taking the limit as $n \to \infty$ yields \[ f(y) \leq \overline{f}(y) + C{(\|V\|_{C^{0,1}}+1)}^{C}|y|^{1/2}\log^2|y|. \]
    \end{proof}

    \section{Homogenization}
    In this section, we prove our main homogenization results for the shape of the reachable set and for solutions of the G equation, using our bounds on convergence of first-passage time to the large-scale average.

    \subsection{The reachable set}
    We combine the random fluctuation bound and nonrandom bias bound to deduce a rate of convergence of the rescaled reachable sets.
    \begin{prop}\label{prop:reachable-set-homog}
        Assume that the law of $V$ is $\Z^d$-translation invariant and that \[ |\div V| \leq C^{-1}{(\|V\|_{C^{0,1}}+1)}^{-C}. \] Then there is a closed set $\mathcal{S} \subseteq \R^d$ such that, for all $t \geq 0$, \[ \P\left[\dist_H(\mathcal{R}_t(0), t\mathcal{S}) > C{(\|V\|_{C^{0,1}}+1)}^{C}t^{1/2}\log^2 t + \lambda\right] \leq C\exp\left(\frac{-C^{-1}{(\|V\|_{C^{0,1}}+1)}^{-C}\lambda^{1/2}}{t^{1/4}}\right), \] where $\dist_H$ denotes the Hausdorff distance. Furthermore, there is a random variable $T_0$, with \[ \E[\exp(C^{-1}{(\|V\|_{C^{0,1}}+1)}^{C}\log^{3/2} T_0)] < \infty, \] such that \[ \sup_{(t,x) \in [0, T] \times B_T} \frac{\dist_H(\mathcal{R}_t(x), x + t\mathcal{S})}{T^{1/2}\log^2 T} \leq C{(\|V\|_{C^{0,1}}+1)}^{C} \] for all $T \geq T_0$.
    \end{prop}
    \begin{proof}
        For the first claim, let $t \geq 0$. Apply Theorem~\ref{cont-estimate} to $B_{(1+\|V\|_{L^\infty})t + 1}$ and Proposition~\ref{prop:random-fluc} to every $x \in \Z^d \cap \overline{B_{(1+\|V\|_{L^\infty})t}}$ and use the union bound to see that as long as $\lambda \geq C{(\|V\|_{C^{0,1}}+1)}^{C}t^{1/2}\log^2 t$ we have
        \begin{equation}\label{random-fluc-applied}
            \P\left[\forall x \in \overline{B_{(1+\|V\|_{L^\infty})t}} :\: \left|\theta(0, x) - \E[\theta(0, x)]\right| > \lambda \right] \leq C\exp\left(\frac{-C^{-1}{(\|V\|_{C^{0,1}}+1)}^{-C}\lambda^{1/2}}{t^{1/4}}\right),
        \end{equation}
        where we absorbed polynomials into the exponential by enlarging the constant $C$. Note also that Theorem~\ref{cont-estimate} implies that if $0 \leq r \leq s$, then
        \begin{equation}\label{cont-estimate-applied}
            \P\left[\mathcal{R}_r(0) \nsubseteq \mathcal{R}_s(0) + B_{\lambda}\right] \leq C\exp\left(\frac{-C^{-1}{(\|V\|_{C^{0,1}}+1)}^{-C}\lambda^{1/2}}{t^{1/4}}\right).
        \end{equation}
        This bounds the random error. On the other hand, Proposition~\ref{prop:scalingbias} shows that
        \begin{equation}\label{nonrandom-bias-applied}
            0 \leq \E[\theta(0, x)] - \overline{\theta}(x) \leq C{(\|V\|_{C^{0,1}}+1)}^{C}|x|^{1/2}\log^2|x|,
        \end{equation}
        where \[ \overline{\theta}(x) := \lim_{\varepsilon \to 0^+} \varepsilon \E[\theta(0, \varepsilon^{-1}x)]. \]
        We define $\mathcal{S} := \{x \in \R^d \mid \overline{\theta}(x) \leq 1\}$. The estimates~(\ref{random-fluc-applied}) and~(\ref{nonrandom-bias-applied}) combine to say that, with high probability, the first passage time $\theta(0, x)$ from $0$ to any point $x$ is close to the large-scale average $\overline{\theta}(x)$. Furthermore, the estimate~(\ref{cont-estimate-applied}) says that once a controlled path reaches $x$, the reachable set stays close to $x$ for all later times (the controllability estimate guarantees the existence of controlled paths in the form of short loops). Unwrapping the definition of Hausdorff distance, along with the fact that $\overline{\theta}$ is positively homogeneous of degree one, i.e. $\theta(tx) = t\theta(x)$ for $t \geq 0$, yields the first claim.

        For the second claim, apply the first claim to every $(t, x) \in (\Z \cap [0, T]) \times (\Z^d \cap B_T)$ and the union bound to conclude that
        \begin{align*}
            \P\left[\sup_{t \in \Z \cap [0, T]} \sup_{x \in \Z^d \cap B_T} \dist_H(\mathcal{R}_t(x), x + t\mathcal{S}) > C{(\|V\|_{C^{0,1}}+1)}^{C}T^{1/2}\log^2 T + \lambda\right]\\
            \leq CT^{d+1}\exp\left(\frac{-C^{-1}{(\|V\|_{C^{0,1}}+1)}^{-C}\lambda^{1/2}}{T^{1/4}}\right).
        \end{align*}
        Next, apply the controllability estimate in $B_T$ to see that the same holds for all $(t, x) \in [0, T] \times B_T$, by enlarging the constant $C$. Plugging in $\lambda = CT^{1/2}\log^2 T$ shows that
        \begin{align*}
            \P\left[\sup_{(t, x) \in [0, T] \times B_T} \frac{\dist_H(\mathcal{R}_t(x), x+t\mathcal{S})}{T^{1/2}\log^2 T} > C{(\|V\|_{C^{0,1}}+1)}^{C}\right]\\
            \leq C\exp(-C^{-1}{(\|V\|_{C^{0,1}}+1)}^{-C}\log^{3/2} T),
        \end{align*}
        and the conclusion follows.
    \end{proof}

    \subsection{Solutions of the G equation}
    We now turn to the proof of Theorem~\ref{thm:main-homog}.
    \begin{proof}
        Let $u^\varepsilon$ be a solution to the G equation~\eqref{eq:eps-G} with initial data $u_0$, and let $\overline{u}$ be the solution the the effective equation~\eqref{eq:macro-G} with the same initial data. The effective Hamiltonian is given by
        \begin{equation}\label{eq:macro-H}
            H(p) := \sup_{v \in \mathcal{S}} p \cdot v
        \end{equation}
        The optimal control formulations are
        \begin{equation}\label{eq:eps-oc}
            u^\varepsilon(t, x) = \sup_{\varepsilon \mathcal{R}_{\varepsilon^{-1}t}(\varepsilon^{-1}x)} u_0
        \end{equation}
        and
        \begin{equation}\label{eq:macro-oc}
            u^\varepsilon(t, x) = \sup_{x+t\mathcal{S}} u_0
        \end{equation}
        respectively.

        Using the representation formulas~(\ref{eq:macro-oc}) and~(\ref{eq:eps-oc}), we see that for every $0 \leq t \leq T$ and $x \in B_T$ we have \[ |u^\varepsilon(t, x) - \overline{u}(t, x)| = \left|\sup_{\varepsilon \mathcal{R}_{\varepsilon^{-1}t}(\varepsilon^{-1}x)} u_0 - \sup_{x + t\mathcal{S}} u_0\right| \leq \Lip(u_0)\dist_H(\varepsilon\mathcal{R}_{\varepsilon^{-1}t}(\varepsilon^{-1}x), x + t\mathcal{S}). \] Rescaling by $\varepsilon^{-1}$ and applying Proposition~\ref{prop:reachable-set-homog} yields \[ \sup_{(t, x) \in [0, T] \times B_T} \dist_H(\varepsilon\mathcal{R}_{\varepsilon^{-1}t}(\varepsilon^{-1}x), x + t\mathcal{S}) \leq C{(\|V\|_{C^{0,1}}+1)}^{C}{(T\varepsilon)}^{1/2}\log^2 (\varepsilon^{-1}T) \] for all $T \geq \varepsilon T_0$, and the result follows.
    \end{proof}

    \bibliographystyle{plain}
    \bibliography{main}
\end{document}